\documentclass[11pt,reqno]{amsart}

\usepackage{amssymb}
\usepackage{amsfonts}
\usepackage{mathrsfs}
\usepackage{cite}

\newcommand{\rmv}[1]{}
\def\Gpi{\pv G_{\pi}}
\def\barGpi{\ov{\pv G}_{\pi}}
\def\Hop{B}
\def\CP{CP_{\pi}(T)}
\def\flag{\mathscr{F}(S)}
\def\flagc{\ov{\mathscr{F}}(S)}
\def\flaga{\mathscr{F}_{\pi}(S)}
\def\flagca{\ov{\mathscr{F}}_{\pi}(S)}

\def\wh{\widehat}

\def\pv#1{\ensuremath{{\mathbf{#1}}}}

\def\inv{^{-1}}
\def\p{\varphi}

\def\pinv{{\p \inv}}
\def\J{\mathrel{{\mathscr J}}} 

\def\R{\mathrel{{\mathscr R}}} 
\def\L{\mathrel{{\mathscr L}}} 
\def\H{\mathrel{{\mathscr H}}} 

\def\e<{\leq _{E}}

\def\ov#1{\ensuremath{\overline {#1}}}

\def\til#1{\ensuremath{\widetilde {#1}}}

\def\malce{\mathbin{\hbox{$\bigcirc$\rlap{\kern-8.75pt\raise0,50pt\hbox{$\mathtt{m}$}}}}}

\def\1sk{^{(1)}}

\def\to{\rightarrow}

\def\stab#1{\mathrm{St}_R(#1)}

\def\pl#1#2{\mathsf{PL}_{\pv {#1}}(#2)}
\def\red{\rho}
\def\bigwr{\check S^{\infty}}
\def\bigwrf{\check S^{\infty}_0}
\def\suppera{S^{\pi}}

\def\Thmname{Theorem}
\def\Propname{Proposition}
\def\Lemmaname{Lemma}
\def\Definitionname{Definition}
\newtheorem{Thm}{\Thmname}[section]
\newtheorem{Prop}[Thm]{\Propname}
\newtheorem{Lemma}[Thm]{\Lemmaname}

\newtheorem{Cor}[Thm]{Corollary}

\newtheorem{Claim}{Claim}
\newtheorem{Subclaim}{Subclaim}

\theoremstyle{remark}
\newtheorem{Example}[Thm]{Example}
\newtheorem{Rmk}[Thm]{Remark}
\theoremstyle{definition}
\newtheorem{Def}[Thm]{\Definitionname}
\numberwithin{equation}{section}

\title{Aperiodic Pointlikes and Beyond}

\author{Karsten Henckell\and John Rhodes\and Benjamin Steinberg}
\address{Department of Mathematics/Computer Science \\ New College of Florida
5800 Bay Shore Road Sarasota, Florida 34243-2109\\ \and
Department of Mathematics\\
University of California at Berkeley \\
Berkeley \\ CA 94720\\
USA\\ \and School of Mathematics and Statistics\\
Carleton University \\
1125 Colonel By Drive\\
Ottawa, Ontario  K1S 5B6 \\
Canada}
\thanks{The third author was supported in part by NSERC} \email{bsteinbg@math.carleton.ca}

\date{June 2, 2007}

\dedicatory{Dedicated to the memory of Bret Tilson}


\begin{document}
\begin{abstract}
We prove that if $\pi$ is a recursive set of primes, then pointlike
sets are decidable for the pseudovariety of semigroups whose
subgroups are $\pi$-groups. In particular, when $\pi$ is the empty
set, we obtain Henckell's decidability of aperiodic pointlikes.  Our
proof, restricted to the case of aperiodic semigroups, is simpler
than the original proof.
\end{abstract}
\maketitle

\section{Introduction}
In~\cite{Henckell} the first author showed that aperiodic pointlikes
are computable;  the companion result for groups was proved by
Ash~\cite{Ash}.  Recently there has been renewed interest in the
decidability of aperiodic pointlikes: the third author used it to
compute certain joins~\cite{slice,slice2}; the authors have recently
used it to study aperiodic idempotent pointlikes and stable
pairs~\cite{Henckellidem,Henckellstable,ourstablepairs}. As a
consequence, the Mal'cev product $\pv V\malce \pv A$ is always
decidable if $\pv V$ is decidable and the semidirect product $\pv
V\ast \pv A$ is decidable so long as $\pv V$ is local and decidable.
The original proof of the decidability of aperiodic pointlikes
in~\cite{Henckell} is quite long.  The key complication is that the
aperiodic semigroup used to compute the pointlikes is given in terms
of generators of a transformation semigroup. To prove the semigroup
is aperiodic, the first author used a complicated Zeiger coding of
the Rhodes expansion to show that these generators live inside a
wreath product of aperiodic semigroups, and hence generate an
aperiodic semigroup. An alternate approach was given by the first
author in~\cite{productexp} involving a simpler coding into a wreath
product.

We prove here a considerable generalization of this result.  If
$\pi$ is a set of primes, let $\Gpi$ denote the pseudovariety of
groups with order divisible only by primes in $\pi$, that is, the
pseudovariety of $\pi$-groups. Then $\barGpi$ denotes the
pseudovariety of semigroups whose subgroups are $\pi$-groups. For
instance, when $\pi=\emptyset$, then $\barGpi$ is the pseudovariety
of aperiodic semigroups; if  \mbox{$\pi=\{p\}$}, then $\barGpi$ is
the pseudovariety of semigroups whose subgroups are $p$-groups.
Notice that $\barGpi$ has decidable membership if and only if $\pi$
is recursive.  We prove in this case that $\barGpi$ has decidable
pointlikes.  Our construction is inspired by Henckell's
proof~\cite{Henckell}, but we sidestep the Zeiger coding by working
directly with $\L$-chains.

 The paper is organized as follows. Given a
finite semigroup $T$, we first introduce a certain computable
semigroup $\CP$ of $\barGpi$-pointlikes. Then we discuss
Sch\"utzenberger groups and the notion of a $\pi'$-free element. In
the following section, we show how to associate a finite semigroup
$S^{\pi}\in \barGpi$ to any finite semigroup $S$. By working with an
arbitrary semigroup and axiomatizing the essential properties of
Henckell's original construction, we manage to simplify Henckell's
proof scheme.  We draw inspiration from the Grigorchuk school's
theory of self-similar (or automaton) groups~\cite{GNS,selfsimilar}.
The subsequent section shows how to construct a relational morphism
from $T$ to $\CP^{\pi}$ such that the inverse image of each element
belongs to $\CP$. The construction assumes the existence of a blowup
operator on $\CP$, the existence of which is established in the
final section.  This last bit again simplifies the corresponding
construction in~\cite{Henckell}.

\section{Pointlikes}
As usual, if $S$ is a finite semigroup, then $s^{\omega}$ denotes as
usual the unique idempotent power of $s$. The notation $S^I$ stands
for $S$ with an adjoined identity $I$.   The reader is referred
to~\cite{Almeida:book,CP,Arbib,Eilenberg,qtheor} for background and
undefined terminology concerning finite semigroups.

\begin{Def}[$\pv V$-pointlikes]
Let $\pv V$ be a pseudovariety of semigroups and $T$ a semigroup.  A
subset $Z\subseteq T$ is said to be \emph{$\pv V$-pointlike} if, for
all relational morphisms $\p:T\to S$ with $S\in \pv V$, there exists
$s\in S$ such that $Z\subseteq s\pinv$.
\end{Def}

The collection $\pl V T$ of $\pv V$-pointlikes of $T$ is a
subsemigroup of the power set $P(T)$, containing the singletons, and
which is a downset for the order $\subseteq$.  One says that $\pv V$
has decidable pointlikes if one can effectively compute $\pl V T$
from the multiplication table of $T$.
See~\cite{Almeidahyp,Henckell,slice,slice2,delay,qtheor} for more on
pointlikes. If $Z\in P(T)$, let us define $Z^{\omega+\ast} =
Z^{\omega}\bigcup_{n\geq 1} Z^n$. Since products distribute over
union in $P(T)$, it follows easily that
\begin{equation}\label{eq:omega+*}
ZZ^{\omega+\ast}=Z^{\omega+\ast}=Z^{\omega+\ast}Z.
\end{equation}
One deduces immediately from \eqref{eq:omega+*} that
$Z^{\omega+\ast}$ is an idempotent. Observe that if $Z$ is a group
element, then $Z^{\omega+\ast}=\bigcup_{n\geq 1} Z^n$.

Let $\pi$ be a set of primes; then $\pi'$ denotes the set of primes
not belonging to $\pi$. Denote by $\Gpi$ the pseudovariety of
$\pi$-groups, that is, groups whose orders only involve primes from
$\pi$. Let $\barGpi$ be the pseudovariety of semigroups whose
subgroups are $\pi$-groups. As mentioned in the introduction, if
$\pi = \emptyset$, then $\Gpi$ is the trivial pseudovariety and
$\barGpi$ is the pseudovariety of aperiodic semigroups; if
$\pi=\{p\}$, then $\Gpi$ is the pseudovariety of $p$-groups. Notice
that the membership problems for $\pi$, $\Gpi$ and $\barGpi$ are
equivalent.

The following proposition shows that the semigroup of
$\barGpi$-pointlikes is closed under unioning up cyclic
$\pi'$-subgroups.  If $\pi$ is empty, this means one can union up
any cyclic subgroup, as was observed by Henckell~\cite{Henckell}. In
fact, $\pl A T$ is closed under the operation $Z\mapsto
Z^{\omega+\ast}$.

\begin{Prop}[Cyclic amalgamation]\label{closedunderomega+star}
Let $\pi$ be a set of primes and $T$ a finite semigroup.  Suppose
$Z\in \pl {\barGpi} T$ generates a cyclic $\pi'$-group. Then
$Z^{\omega+\ast}\in \pl {\barGpi} T$.
\end{Prop}
\begin{proof}
Suppose the group element $Z$ has order $k$. Let $\p:T\to S$ be a
relational morphism with $S\in \barGpi$. Choose $s\in S$ with
$Z\subseteq s\pinv$.  Then $Z=Z^{\omega}Z\subseteq
s^{\omega}s\pinv$.  So without loss of generality, we may assume
that $s$ is a group element.  The order $n$ of $s$ must be prime to
$k$, so we can find an positive integer $m$ with $mn\equiv 1\bmod
k$.  Then $Z = Z^{mn}\subseteq s^{mn}\pinv = s^{\omega}\pinv$.  Thus
$Z^r\subseteq s^{\omega}\pinv$ for all $r>0$ and so
$Z^{\omega+\ast}\subseteq s^{\omega}\pinv$. We conclude
$Z^{\omega+\ast}\in \pl {\barGpi} T$.
\end{proof}

This paper is devoted to proving the following generalization of
Henckell's theorem describing the $\pv A$-pointlike
sets~\cite{Henckell}.

\begin{Thm}\label{henckellmain}
Let $\pi$ be a set of primes and $T$ be a finite semigroup. Denote
by $\CP$ the smallest subsemigroup of $P(T)$ containing the
singletons and closed under $Z\mapsto Z^{\omega+\ast}$ whenever $Z$
generates a cyclic $\pi'$-group. Then $\pl {\barGpi} T$ consists of
all $X\in P(T)$ with $X\subseteq Y$ some $Y\in \CP$.
\end{Thm}

Proposition~\ref{closedunderomega+star} shows that $\CP\subseteq \pl
{\barGpi} T$ and hence each of the subsets described in
Theorem~\ref{henckellmain} is indeed $\barGpi$-pointlike.  The hard
part of the result is proving the converse.

\begin{Cor}
Let $\pi$ be a recursive set of prime numbers.  Then
$\barGpi$-pointlikes are decidable.
\end{Cor}

In~\cite{Henckellidem,ourstablepairs} it is shown that if $\pv V$ is
a pseudovariety such that $\pv A\malce \pv V=\pv V$ and  $\pv
V$-pointlikes are decidable, then the $\pv V$-idempotent pointlikes
are decidable.  This in particular applies to pseudovarieties of the
form $\barGpi$.

\begin{Cor}
Let $\pi$ be a recursive set of prime numbers.  Then
$\barGpi$-idempotent pointlikes are decidable and hence the Mal'cev
product $\pv V\malce \barGpi$ is decidable whenever $\pv V$ is
decidable.
\end{Cor}

\section{Sch\"utzenberger groups and $\pi'$-free
elements}\label{schutzsec} Fix a semigroup $S$. Let $H$ be an
$\H$-class of $S$ and set \[\stab H = \{s\in S^I\mid Hs\subseteq
H\}.\]  The faithful quotient transformation monoid, denoted
$(H,\Gamma_R(H))$, is a transitive regular permutation group called
the \emph{Sch\"utzenberger group} of $H$~\cite{CP,Arbib}.  If $H$ is
a maximal subgroup, then $\Gamma_R(H)=H$. In general, one can always
find a subgroup $\til \Gamma_R(H)\subseteq \stab H$ acting
transitively on $H$ with faithful quotient
$\Gamma_R(H)$~\cite{Arbib}.  The following proposition is
well-known.

\begin{Prop}\label{L-classindep}
Let $H$ and $H'$ be $\L$-equivalent $\H$-classes of $S$.  Then
$\stab H=\stab {H'}$ and one can take $\til \Gamma_R(H)= \til
\Gamma_R(H')$. Moreover, the kernels of the natural maps
$\til\Gamma_R(H)\to \Gamma_R(H)$ and \mbox{$\til\Gamma_R(H)\to
\Gamma_R(H')$} coincide. In particular, $\Gamma_R(H)\cong
\Gamma_R(H')$.
\end{Prop}
\begin{proof}
Suppose that $H=H_a$, $H'=H_b$ and $ya=b$ with $y\in S^I$. By
Green's lemma, $yH_a=H_b$.  The equality $\stab H=\stab {H'}$ is
classical~\cite{CP,Arbib}. Now $H_a=a\til \Gamma_R(H)$, so $b\til
\Gamma_R(H) = ya\til \Gamma_R(H) = yH_a=H_b$.  Thus $\til
\Gamma_R(H)$ is a group in $\stab {H'}$ acting transitively on $H$.
By regularity of the action we can take $\til\Gamma_R(H') = \til
\Gamma_R(H)$. The statement about kernels follows since the right
stabilizer of any two $\L$-equivalent elements of a semigroup
coincide.
\end{proof}

Similarly, there is a left Sch\"utzenberger group $(\Gamma_L(H),H)$
and a subgroup $\til \Gamma_L(H)$ of the left stabilizer of $H$
mapping onto $\Gamma_L(H)$.  The groups $\Gamma_L(H)$ and
$\Gamma_R(H)$ are isomorphic.  In fact, if $h_0\in H$ is a fixed
base point and $g\in \Gamma_R(H)$, then the map $\gamma$ sending $g$
to the unique $g\gamma\in \Gamma_L(H)$ with $g\gamma h_0 = h_0 g$ is
an anti-isomorphism.  In particular, using
Proposition~\ref{L-classindep} and its dual, we see that the
Sch\"utzenberger group depends up to isomorphism only on the
$\J$-class.  See~\cite{CP,Arbib} for details. The following
proposition describes when an element belongs to $\stab H$, and
hence represents an element of $\Gamma_R(H)$.

\begin{Prop}\label{stayinschutz}
Let $H$ be an $\H$-class of $S$.  Then $s\in S^I$ belongs to $\stab
H$ if and only if, for some $h\in H$, $hs\in H$.
\end{Prop}
\begin{proof}
Necessity is clear.  Suppose $hs\in H$ and $h'\in H$.  Then we have
$h's\L hs\L h'$.  Therefore, $h's\J h'$ and so $h's\R h'$.  This
shows $h's\in H$.
\end{proof}

We now introduce the important notion of $\pi'$-freeness.

\begin{Def}[$\pi'$-free]\label{definepiprimefree}
Let $\pi$ be a set of primes. A $\J$-class (respectively, $\L$-,
$\R$-class) of $S$ is called \emph{$\pi'$-free} if its
Schutzenberger group is a $\pi$-group.  Likewise, an element of a
$\pi'$-free $\J$-class is called $\pi'$-free.
\end{Def}

We shall need the following well-known and easy to prove lemma.
\begin{Lemma}\label{primelift}
Let $\p:G\to H$ be an onto group homomorphism and let $h\in H$ have
prime order $p$.  There there is an element $g\in G$ of prime power
order $p^n$ with $g\p=h$.
\end{Lemma}

\section{A $\barGpi$-variant of the Rhodes expansion}
Our goal in this section is to associate a finite semigroup
$\suppera\in\barGpi$ to each finite semigroup $S$.  The case of
$\CP$ will yield a semigroup in $\barGpi$ and a relational morphism
that establishes Theorem~\ref{henckellmain}.

Fix a finite semigroup $S$ for this section.
 Elements of the free monoid $S^*$ will be written as
strings $\vec{x}=(x_n,x_{n-1},\ldots,x_1)$. The empty string is
denoted $\varepsilon$.  We omit parentheses for strings of length
$1$.  If $n\geq \ell$,  define
\begin{equation*}\label{defalpha}
(x_n,x_{n-1},\ldots,x_1)\alpha_{\ell} =
(x_{\ell},x_{\ell-1},\ldots,x_1)
\end{equation*}
and $(x_n,\ldots,x_1)\tau_{\ell} = x_{\ell}$. We identify $\vec
x\alpha_1$ with $\vec x\tau_1$, the first letter of $\vec x$.  By
convention $\vec x\alpha_0=\varepsilon$.   Set
$(x_n,\ldots,x_1)\omega=x_n$. We use $\vec b\cdot\vec a$ for the
concatenation of $\vec b$ and $\vec a$. As the notation suggests, we
 read strings from right to left.

  If $P$ is a pre-ordered set,
then a \emph{flag} of elements of $P$ is a strict chain
$p_n<p_{n-1}<\cdots<p_1$. We also allow an empty flag.  Denote by
$\flag$ the set of flags for the $\L$-order on $S$.  Of course,
$\flag$  is a finite set.  A typical flag
$s_n<_{\L}s_{n-1}<_{\L}\cdots <_{\L} s_1$ shall be denoted
$(s_n,s_{n-1},\cdots,s_1)$.  We shall also consider the set $\flagc$
of $\L$-chains, that is, all strings $(s_n,s_{n-1},\ldots,s_1)\in
S^*$ (including the empty string) such that $s_{i+1}\leq_{\L}s_i$
for all $i$.  Of course $\flag\subseteq \flagc\subseteq S^*$. A
string $(s_n,\ldots,s_1)$ is termed \emph{$\pi'$-free} if each $s_i$
is $\pi'$-free (see Definition~\ref{definepiprimefree}).

We use $\flaga$ and $\flagca$ to denote the respective subsets of
$\flag$ and $\flagc$ consisting of $\pi'$-free strings.  There is a
natural retraction from $\flagc$ to $\flag$ (mapping $\flagca$ onto
$\flaga$), which we proceed to define. Define an elementary
reduction to be a rule of the from $(s',s)\to s'$ where $s'\L s$.
Elementary reductions are length-decreasing.  It is well known and
easy to prove that the elementary reductions form a confluent
rewriting system and so each element $\vec{x}\in \flagc$ can be
reduced to a unique flag $\vec{x}\mathcal \red\in \flag$, called its
\emph{reduction}~\cite{TilsonXII}. Clearly the reduction map is a
retract and takes $\flagca$ to $\flaga$. The Rhodes
expansion~\cite{TilsonXII} defines a multiplication on $\flag$ using
the reduction map. Our constructions are motivated by properties of
the Rhodes expansion, but we shall not need this expansion per se.
Some key properties of the reduction map, which are immediate from
the definition, are recorded in the following lemma.

\begin{Lemma}\label{reductionmap}
The reduction map $\red$ enjoys the following properties:
\begin{enumerate}
\item For $\vec x\in \flagc$, $\vec x\red\omega = \vec
x\omega$;
\item Let $\vec b,\vec a,\vec b\cdot\vec a\in \flagc$ and suppose $\vec a\red =
(a_{\ell},a_{\ell-1},\ldots,a_1)$.  Then one has $(\vec b\cdot\vec
a)\red\alpha_{\ell} = (x_{\ell},a_{\ell-1},\ldots,a_1)$ where
$x_{\ell}\L a_{\ell}$.
\end{enumerate}
\end{Lemma}

Let us now turn to defining an auxiliary semigroup that will play a
role in the proof.
\begin{Def}[$\check S$]\label{definechecks}
Denote by $\check S$ the monoid of all functions $f:S\to S$ such
that
\begin{enumerate}
\item  $sf\leq_{\R} s$ for all $s\in S$;
\item  $f$ preserves $\L$, i.e.\ $s\L s'$ implies $sf\L s'f$;
\item  There exists $s_f\in S^I$ such that $sf\R s$ implies $sf
=ss_f$.
\end{enumerate}
\end{Def}
 Notice that the natural
action of $S^I$ on the right of $S$ belongs to $\check S$.

\begin{Prop} The set $\check S$ is a monoid.
\end{Prop}
\begin{proof}
Clearly $\check S$ contains the identity. The set of functions
satisfying the second item is obviously closed under composition.
Suppose $f,g\in S$. Then  $sfg\leq_{\R} sf\leq_{\R} s$. Moreover, if
$s\R sfg$, all the inequalities are equalities and so $sfg = sfs_g
=ss_fs_g$. In particular, we can take $s_{fg}=s_fs_g$.
\end{proof}

Let us write $\check S^{\infty}$ for the action monoid of the
infinite wreath product $\wr^{\infty} (S,\check S)$ of right
transformation monoids $(S,\check S)$. There is a natural action of
$\bigwr$ on $S^*$ by length-preserving, sequential functions via the
projections $\wr^{\infty}(S,\check S)\to \wr^n(S,\check S)$;  to
obtain the action on a word of length $n$, project first to
$\wr^n(S,\check S)$ and then act.  If $F\in \bigwr$ and $\vec a\in
S^*$, then there is a unique element ${}_{\vec a}F\in \check
S^{\infty}$ such that $(\vec b\cdot \vec a)F = \vec b{}_{\vec
a}F\cdot \vec aF$ for all $\vec b\in S^*$; for example,
${}_{\varepsilon}F = F$.  Also if $F\in \bigwr$, then there is also
a unique element $\sigma_F\in \check S$ such that, for $s\in S$, the
equality $sF = s\sigma_F$ holds.  In particular,
\begin{equation}\label{wreathrecursions}
(s_n,\ldots,s_1)F = (s_n,\ldots,s_2){}_{s_1}F\cdot s_1\sigma_F.
\end{equation}
So $\sigma_F$ describes the action on the first letter, and must
belong to $\check S$, while ${}_{s_1}F\in \check S^{\infty}$ is how
$F$ acts on the rest of a string starting with $s_1$.   In fact,
\eqref{wreathrecursions} can serve as a recursive definition of what
it means to belong to $\bigwr$ (c.f.~\cite{GNS,selfsimilar}).  Now
in our situation, by definition of $\check S$, there is an element
$s_{\sigma_F}\in S^I$ so that $s_1F =s_1\sigma_F\R s_1$ implies
$s_1F = s_1s_{\sigma_F}$.

Let us consider some examples to illustrate this formalism for
infinite iterated wreath products, in particular the wreath
recursion \eqref{wreathrecursions}.  For simplicity, we work with
$\wr^{\infty} (\{0,1\},T_2)$ where $T_2$ is the full transformation
semigroup on two letters. Notice that the iterated wreath product
$\wr^{\infty} (\{0,1\},T_2)$ is isomorphic to the semidirect product
$(\wr^{\infty} (\{0,1\},T_2))^2\rtimes T_2$.  As a first example
example, consider the $2$-adic odometer, which adds one to the
$2$-adic expansion of an integer (where the least significant bit is
the first one read from \emph{right to
left})~\cite{GNS,selfsimilar}.

\begin{Example}[Odometer]
Let $A$ be the $2$-adic odometer considered above, acting on
$\{0,1\}^*$, and let $I$ be the identity function on $\{0,1\}^*$. If
a $2$-adic integer has $0$ as its least significant bit, we change
the $0$ to a $1$ and then continue with the identity map the rest of
the way; if the least significant bit is $1$, we change it to $0$
and we add $1$ to what remains (i.e.\ perform a carry).  So in terms
of the wreath recursion \eqref{wreathrecursions}, $\sigma_A = (01)$
and ${}_0A = I$, ${}_1A=A$. If we identify $\wr^{\infty}
(\{0,1\},T_2)$ with $\wr^{\infty} (\{0,1\},T_2)^2\rtimes T_2$, then
$A=((I,A),(01))$.
\end{Example}

Next we consider the two sections to the unilateral shift.

\begin{Example}[Shift]
Consider the functions $F$, $G$ on $\{0,1\}^*$ that send
$x_nx_{n-1}\cdots x_1$ to, respectively, $x_{n-1}x_{n-2}\cdots x_10$
and $x_{n-1}x_{n-2}\cdots x_11$.  Both of these functions act by
remembering the first letter, then resetting it to a predetermined
symbol, and then resetting the second letter to the first and so on
and so forth. Formally, the wreath recursion
\eqref{wreathrecursions} is given by $\sigma _{F} = \ov 0$,
$\sigma_G=\ov 1$ (where $\ov x$ is the constant map to $x$) and
${}_0F=F={}_0G$, ${}_1F=G={}_1G$. Identifying $\wr^{\infty}
(\{0,1\},T_2)$ with $\wr^{\infty} (\{0,1\},T_2)^2\rtimes T_2$, we
have $F=((F,G),\ov 0)$ and $G= ((F,G),\ov 1)$.  So, for example, the
wreath recursion
\[(x_nx_{n-1}\cdots x_21)F = (x_nx_{n-1}\cdots x_2)G\cdot 0=x_{n-1}x_{n-2}\cdots
x_210\] holds. Notice that on infinite bit strings, $F$ and $G$ are
the two sections to the unilateral shift that erases the first
letter.
\end{Example}

From these examples, the reader should instantly see the connection
between iterated wreath products and sequential
functions~\cite{Eilenberg,EilenbergA,GNS,selfsimilar}.

A subsemigroup $T$ of $\bigwr$ is called \emph{self-similar} if, for
all $F\in T$ and $\vec a\in S^*$, one has ${}_{\vec a}F\in T$; so
$\bigwr$ itself is self-similar.  It is actually enough that, for
each letter $s\in S$, one has ${}_sF\in T$.   For instance, the
group generated by the $2$-adic odometer $A$ is self-similar in
$\wr^{\infty} (\{0,1\},T_2)$ since ${}_0A=I$, ${}_1A=A$.  Similarly
the semigroup generated by the two sections $F$, $G$ to the shift is
self-similar since ${}_0F=F={}_0G$, ${}_1F=G={}_1G$.  This viewpoint
on infinite wreath products is due to Grigorchuk and
Nekrashevych~\cite{GNS,selfsimilar}.

\begin{Def}[$\bigwrf$]\label{defbigwrf}
Denote by $\bigwrf$ the collection of all transformations $F\in
\bigwr$ such that whenever $(x_n,x_{n-1},\ldots,x_1)F =
(y_n,y_{n-1},\ldots, y_1)$ with \mbox{$x_{n-1}\R y_{n-1}$} and
$x_n\R y_n$,  there exists $s\in S^I$ with $x_{n-1}s = y_{n-1}$ and
$x_ns=y_n$.
\end{Def}

The element $s$ can depend on the string $(x_n,\ldots,x_1)$.

\begin{Prop}
$\bigwrf$ is a self-similar submonoid of $\bigwr$.
\end{Prop}
\begin{proof}
Clearly it contains the identity.    Suppose $F,G\in \bigwrf$ and
\[(x_n,x_{n-1},\ldots,x_1)FG = (z_n,z_{n-1},\ldots,z_1)\] with
$x_{n-1}\R z_{n-1}$ and $x_n\R z_n$. Suppose that
$(x_n,x_{n-1},\ldots,x_1)F = (y_n,y_{n-1},\ldots,y_1)$. Then, for
$i=n-1,n$, we have $x_i\geq_{\R} y_i\geq_{\R} z_i\R x_i$. Thus
$x_i\R y_i$ and $y_i\R z_i$, $i=n-1,n$.  By assumption, there exists
$s,t\in S$ with $x_is=y_i$ and $y_it=z_i$, $i=n-1,n$.  Then
$x_ist=z_i$, for $i=n-1,n$. Hence $\bigwrf$ is submonoid of
$\bigwr$. Self-similarity is immediate from the equation
$(x_n,\ldots,x_1){}_{\vec a}F\cdot \vec aF = ((x_n,\ldots,x_1)\cdot
\vec a)F$ and the definition of $\bigwrf$.
\end{proof}

If $s\in S$, define the diagonal operator $\Delta_s:S^*\to S^*$ by
\begin{equation}\label{diagonal}
(x_n,x_{n-1},\ldots,x_1)\Delta_s = (x_ns,x_{n-1}s,\ldots,x_1s).
\end{equation}
 It
is immediate  $\Delta_s\in \bigwrf$.  The next lemma expresses the
so-called Zeiger property of $\bigwrf$.

\begin{Lemma}[Zeiger Property]\label{zeigerprop}
Suppose $\vec x = (x_n,\ldots,x_1)\in \flagc$ and $F\in \bigwrf$ is
such that $\vec xF = (y_n,y_{n-1},\ldots,y_1)$ with
$x_{n-1}=y_{n-1}$ and $x_n\R y_n$.  Then $x_n=y_n$.
\end{Lemma}
\begin{proof}
By definition of $\bigwrf$, there exists $s\in S^I$ such that
$x_{n-1}s = y_{n-1}=x_{n-1}$ and $x_ns=y_n$.  Since
$x_n\leq_{\L}x_{n-1}$, we can write $x_n= ux_{n-1}$ with $u\in S^I$.
Then $y_n = x_ns=ux_{n-1}s = ux_{n-1}=x_n$, as required.
\end{proof}

We  now define an important transformation semigroup on $\flagc$.

\begin{Def}[$\mathscr C$]
Let $\mathscr C$ consist of all transformations \mbox{$f:\flagc\to
\flagc$} such that there exists $(\wh f,\ov f)\in \bigwrf\times
(\flagc\setminus \{\varepsilon\})$ with $\vec xf = \vec x\wh f\cdot
\ov f$.
\end{Def}

  For instance, if $s\in
S$, then one readily checks that $(\Delta_s,s)$ defines an element
of $\mathscr C$ via the formula:
\[(x_n,x_{n-1},\ldots,x_1)(\Delta_s,s) = (x_ns,x_{n-1}s,\ldots, x_1s,s).\]
Such elements correspond to generators of the Rhodes
expansion~\cite{TilsonXII}.  Notice that in order for $(\wh f,\ov
f)$ to define an element of $\mathscr C$ one must have $x\wh
f\leq_{\L} \ov f\omega$ for every $x\in S$. It is essential that
$\ov f$ is not empty in the definition of $\mathscr C$.

\begin{Prop}\label{Cissemigroupandformula}
The set $\mathscr C$ is a semigroup.
\end{Prop}
\begin{proof}
Let $f,g\in \mathscr C$.  We claim $\wh{fg} = \wh{f}{}_{\ov
f}\wh{g}$ and $\ov {fg} = (\ov f)\wh g\cdot \ov g$. Indeed,
\begin{equation*}
\vec xfg = (\vec x\wh f\cdot \ov f)g  = (\vec x\wh f\cdot \ov f)\wh
g\cdot \ov g  = \vec x\wh f{}_{\ov f}\wh g\cdot (\ov f)\wh g\cdot
\ov g.
\end{equation*}
As $\bigwrf$ is self-similar, ${}_{\ov f}\wh{g}\in \bigwrf$ and so
$fg\in\mathscr C$.
\end{proof}

An immediate consequence of the definition is:
\begin{Lemma}\label{wreathlike}
If $f\in \mathscr C$ and $\vec a,\vec b,\vec b\cdot\vec a\in
\flagc$, then $(\vec b\cdot \vec a)f = \vec b{}_{\vec a}\wh f\cdot
\vec af$.
\end{Lemma}
\begin{proof}
Indeed, a straightforward computation yields \[(\vec b\cdot \vec a)f
= (\vec b\cdot \vec a)\wh f\cdot \ov f = \vec b{}_{\vec a}\wh f\cdot
(\vec a\wh f\cdot \ov f) = \vec b{}_{\vec a}\wh f\cdot \vec af,\]
proving the lemma.
\end{proof}

Lemma~\ref{wreathlike} shows that $\mathscr C$ behaves very much
like an iterated wreath product, a property that we shall exploit
repeatedly. In fact, an element of $\mathscr C$ is like an
asynchronous transducer that outputs $\ov f$ with empty input and
then computes synchronously.  We are almost prepared to define our
semigroup in $\barGpi$. Recall that $\rho$ denotes the reduction
map.

\begin{Def}[$\mathscr C^{\pi}$]\label{CA}
Let $\mathscr C^{\pi}$ be the subset of $\mathscr C$ consisting of
all transformations $f\in \mathscr C$ such that:
\begin{enumerate}
\item $\flagca f\subseteq \flagca$;
\item $\red f\red =f\red$.
\end{enumerate}
\end{Def}

\begin{Prop}\label{CAissemi}
The set $\mathscr C^{\pi}$ is a semigroup.
\end{Prop}
\begin{proof}
Closure of the first item under composition is clear.  The
computation
\begin{equation}\label{redishom}
\red fg\red = \red f (\red g\red) = (\red f\red)g\red = f\red g\red
= fg\red
\end{equation}
completes the proof that $\mathscr C^{\pi}$ is a semigroup.
\end{proof}

Notice that \eqref{redishom} allows us to define an action of
$\mathscr C^{\pi}$ on $\flaga$ by $\vec x\mapsto \vec xf\rho$ for
$f\in \mathscr C^{\pi}$.  Let us denote the resulting faithful
transformation semigroup by $(\flaga,\suppera)$. Observe that
$\flagca$ appears in Definition~\ref{CA}, while in the definition of
$\suppera$ we use $\flaga$.  Since $\flaga$ is finite, so is
$\suppera$. Our goal is to prove $\suppera\in \barGpi$. We do this
by showing that if $\vec x\in \flaga$, $p\in \pi'$ and $f\in
\suppera$ with $\vec xf^p = \vec x$, then $\vec xf=\vec x$. This
implies that $\suppera$ has no cyclic subgroup of prime order
belonging to $\pi'$ and hence $\suppera\in \barGpi$. First we make a
simple observation.

\begin{Lemma}\label{omegaRdrops}
Let $f\in \mathscr C$ and $\varepsilon\neq\vec x\in \flagc$.  Then
\begin{equation*}
\vec xf\red\omega = \vec xf\omega = \vec x\wh f\omega\leq_{\R} \vec
x\omega.
\end{equation*}
\end{Lemma}
\begin{proof}
By definition of $\mathscr C$, we have $\vec xf\omega = (\vec x\wh
f\cdot \ov f)\omega = \vec x\wh f\omega$.  Since $\wh f\in
\bigwrf\subseteq \bigwr$, we conclude $\vec x\wh f\omega\leq_{\R}
\vec x\omega$. The lemma then follows from Lemma~\ref{reductionmap}.
\end{proof}

Notice if $f\in \mathscr C^{\pi}$, then $\vec xf\red\neq
\varepsilon$, for all $\vec x\in \flagc$.  Indeed, $\ov f\neq
\varepsilon$ implies $\vec xf\red = (\vec x\wh f\cdot \ov f)\red\neq
\varepsilon$.  Finally, we turn to the main result of this section.

\begin{Thm}\label{isaperiodic}
Let $S$ be a finite semigroup.  Then $\suppera\in \barGpi$.
\end{Thm}
\begin{proof}
As $\red:\mathscr C^{\pi}\to \suppera$ is a homomorphism by
\eqref{redishom}, it suffices to prove that if $\vec x\in \flaga$,
$p\in \pi'$ and $f\in \mathscr C^{\pi}$ are such that $\vec xf^p\red
=\vec x$, then $\vec xf\red = \vec x$.   Set $\vec x_i = \vec
xf^i\red$; so $\vec x_0=\vec x_p = \vec x$.  As $f\in \mathscr
C^{\pi}$, the above discussion shows $\vec x_i\neq \varepsilon$ for
all $i$. Suppose that $\vec x=(x_n,x_{n-1},\ldots,x_1)$. We prove
the following critical claim by induction on $\ell$.

\begin{Claim}\label{keyclaim}
For each $1\leq i\leq p$ and each $0\leq \ell\leq n$, we have $|\vec
x_i|\geq \ell$ and $\vec x_i\alpha_{\ell}=\vec x\alpha_{\ell} =
(x_{\ell},x_{\ell-1},\ldots,x_1)$.
\end{Claim}

Let us show how the claim implies the theorem.  As $\vec x_1=\vec
xf\red$, it suffices from the claim to show that $|\vec x_1|=|\vec
x|$.  But Lemma~\ref{omegaRdrops} implies that
\begin{equation}\label{rgoesdown}
\vec x\omega\geq_{\R}\vec x_1\omega\geq_{\R}\cdots\geq_{\R}\vec
x_{p-1}\omega\geq_{\R}\vec x_p\omega = \vec x\omega
\end{equation}
 and so $x_n =
\vec x_1\tau_n\geq_{\L}\vec x_1\omega\R \vec x\omega =x_n$.
Therefore, $\vec x_1\omega\L x_n$ and hence, since $\vec x_1$ is a
flag, we conclude $|\vec x_1|=n$.  This shows that $\vec x=\vec
xf\red$ and completes the proof of Theorem~\ref{isaperiodic} once we
establish the claim.

The claim is trivial for $\ell=0$.  Assume the claim is true for
 $0\leq \ell\leq n-1$.  We prove it for $\ell+1$. Let
$\vec a = \vec x\alpha_{\ell} = (x_{\ell},x_{\ell-1},\ldots, x_1)$
and set $\vec x = \vec b\cdot \vec a$; if $\ell=0$, then $\vec a =
\varepsilon$ and $\vec b = \vec x$.  Then $\vec x_i = \vec b_i\cdot
\vec a$ for some $\vec b_i\in \flaga$, each $1\leq i\leq p$, by the
inductive hypothesis. First we show  $|\vec x_i|\geq \ell+1$.
Indeed, \eqref{rgoesdown} implies $\vec x_i\omega\R x_n<_{\L}
x_{\ell}$ since $\vec x$ is a flag. We conclude $\vec
b_i\neq\varepsilon$, that is $|\vec x_i|\geq \ell+1$, for all $i$.

Let us set $\vec z= \vec af\red$.  First observe that $|\vec z|\geq
\ell$ and $\vec z\alpha_{\ell}=\vec a$.  Indeed, \[\vec b_1\cdot
\vec a=\vec x_1 = \vec xf\red = (\vec b\cdot \vec a)f\red = (\vec
b{}_{\vec a}\wh f\cdot \vec af)\red = (\vec b{}_{\vec a}\wh f\cdot
\vec z)\red\] by Lemma~\ref{wreathlike} and the confluence of
reduction. As each entry of $\vec b{}_{\vec a}\wh f$ is
$<_{\J}$-below $x_{\ell}$ (since $\vec x$ is a flag and ${}_{\vec
a}\wh f\in\bigwrf$), we must have $\vec z\alpha_{\ell} = \vec a$.
There are three cases.

\subsection*{Case 1} Suppose that $|\vec z|\geq\ell+2$.  Then an
application of Lemma~\ref{wreathlike} and the second item of
Lemma~\ref{reductionmap} shows that, for any $\vec y\in \flaga$,
\[(\vec y\cdot \vec a)f\red\alpha_{\ell+1} = (\vec y{}_{\vec a}\wh
f\cdot \vec af)\red\alpha_{\ell+1} = (\vec y{}_{\vec a}\wh f\cdot
\vec z)\rho\alpha_{\ell+1} = \vec z\alpha_{\ell+1}.
\]  In particular, $(\vec x_i)\alpha_{\ell+1}$ is independent of $i$
and so equals $(\vec x)\alpha_{\ell+1}$, as desired.

\subsection*{Case 2} Suppose that $|\vec z|=\ell+1$. Our goal is to show that
$\vec b_i\alpha_1 = x_{\ell+1}$ for all $1\leq i\leq p$.  Set
$s=\vec z\omega$ for convenience. Then we have by
Lemma~\ref{wreathlike} and confluence of the reduction map
\[\vec b_i\cdot \vec a=\vec x_{i} = \vec x_{i-1}f\red =(\vec b_{i-1}\cdot \vec a)f\red=
 ((\vec b_{i-1}){}_{\vec a}\wh f\cdot \vec
af)\red= ((\vec b_{i-1}){}_{\vec a}\wh f\cdot \vec z)\red.\]  So the
second item of Lemma~\ref{reductionmap} allows us to deduce
$\vec b_{i}\alpha_1\L \vec z\omega=s$.
On the other hand, Lemma~\ref{reductionmap} implies $s=\vec
z\omega=\vec af\rho\omega = \vec af\omega$ and so, since $\vec x_if
= (\vec b_i){}_{\vec a}\wh f\cdot \vec af\in\flagc$, we must have
\begin{equation}\label{superequation}
(\vec b_i\alpha_1){}_{\vec a}\wh f\leq_{\L}\vec af\omega = s\L \vec
b_i\alpha_1,\ \text{for all}\ i.
\end{equation}

\subsubsection*{Subcase 1} Suppose $(\vec b_j\alpha_1){}_{\vec a}\wh f\R
\vec b_j\alpha_1$ for some $1\leq j\leq p$.  By definition of
$\check S$ and of the wreath product there is an element $u\in S^I$
so that if $y\in S$ and $y{}_{\vec a}\wh f\R y$, then $y{}_{\vec
a}\wh f=yu$. With this notation, we are assuming $\vec
b_j\alpha_1u\R \vec b_j\alpha_1$. By \eqref{superequation}, $\vec
b_j\alpha_1u=(\vec b_j\alpha_1){}_{\vec a}\wh f\leq_{\L} \vec
b_j\alpha_1$, whence $\vec b_j\alpha_1u\H \vec b_j\alpha_1$. In
particular, $u$ represents an element of $\Gamma_R(H_{\vec
b_j\alpha_1})$ by Proposition~\ref{stayinschutz}.  Notice in the
case of aperiodic pointlikes, this already yields $(\vec
b_j\alpha_1){}_{\vec a}\wh f=\vec b_j\alpha_1u=\vec b_j\alpha_1$ and
so the following subclaim is essentially trivial in the aperiodic
case.

\begin{Subclaim}\label{newclaim}
For $i\geq j$, we have $(\vec b_i\alpha_1){}_{\vec a}\wh f\R \vec
b_i\alpha_1$ and $\vec b_i\alpha_1 = \vec b_j\alpha_1u^{i-j}$.
\end{Subclaim}
We prove the subclaim by induction, the case $i=j$ being by
assumption. Assume it is true for $i$.   Then $(\vec
b_i\alpha_1){}_{\vec a}\wh f\R \vec b_i\alpha_1$ implies $(\vec
b_i\alpha_1){}_{\vec a}\wh f=\vec b_i\alpha_1u$. Since $ (\vec
b_i\alpha_1){}_{\vec a}\wh f\leq_{\L} \vec b_i\alpha_1$ by
\eqref{superequation} and  $(\vec b_i\alpha_1){}_{\vec a}\wh f\R
\vec b_i\alpha_1$ by the induction hypothesis, we conclude $\vec
b_i\alpha_1u = (\vec b_i\alpha_1){}_{\vec a}\wh f\L \vec
b_i\alpha_1\L s$, where the last $\L$-equivalence uses
\eqref{superequation}.
 As $\vec b_i$ is a flag, $\vec b_i\tau_2<_{\L} \vec
b_i\alpha_1$ and so $(\vec b_i){}_{\vec a}\wh f\tau_2<_{\J} \vec
b_i\alpha_1$.  Therefore, $(\vec b_i){}_{\vec a}\wh
f\tau_2<_{\J}(\vec b_i\alpha_1){}_{\vec a}\wh f$. Recalling $\vec z$
is a flag of length $\ell+1$, reduction is confluent and $(\vec
b_i\alpha_1){}_{\vec a}\wh f\L s=\vec z\omega$, we obtain
\begin{equation}\label{longequation}
\vec b_{i+1}\alpha_1 = ((\vec b_i){}_{\vec a}\wh f\cdot \vec
af)\red\tau_{\ell+1} = ((\vec b_i){}_{\vec a}\wh f\cdot \vec
z)\red\tau_{\ell+1}= (\vec b_i\alpha_1){}_{\vec a}\wh f =\vec
b_i\alpha_1u
\end{equation}
Since $\vec b_{i+1}\alpha_1 = \vec b_i\alpha_1u \L \vec
b_i\alpha_1$, the second item of Definition~\ref{definechecks}
implies $(\vec b_{i+1}\alpha_1){}_{\vec a}\wh f\L (\vec
b_i\alpha_1){}_{\vec a}\wh f= \vec b_{i+1}\alpha_1$, where the
equality uses \eqref{longequation}.  Since, $(\vec b_{i+1}\alpha_1)
{}_{\vec a}\wh f\leq_{\R} \vec b_{i+1}\alpha_1$, we conclude $(\vec
b_{i+1}\alpha_1){}_{\vec a}\wh f\R \vec b_{i+1}\alpha_1$, as was
required for the subclaim.  Also, by induction  and
\eqref{longequation}
\[\vec b_{i+1}\alpha_1 = \vec b_i\alpha_1u = \vec
b_j\alpha_1u^{i-j}u = \vec b_j\alpha_1u^{i+1-j},\] where the first
equality uses \eqref{longequation}.  This proves
Subclaim~\ref{newclaim}.

Since $\vec b_j = \vec b_{j+p}$, Subclaim~\ref{newclaim} implies
$\vec b_j\alpha_1 = \vec b_{j+p}\alpha_1 = \vec b_j\alpha_1u^p$.
Since $u$ represents an element of $\Gamma_R(H_{\vec b_j\alpha_1})$
and the Sch\"utzenberger group is a regular permutation group, it
follows $u^p$ represents the identity.  But $\vec b_j\alpha_1$ is
$\pi'$-free and $p\in \pi'$.  We conclude that $u$ represents the
identity of $\Gamma_R(H_{\vec b_j\alpha_1})$ and so $\vec
b_j\alpha_1u=\vec b_j\alpha_1$.  Applying the subclaim, it follows
$\vec b_j\alpha_1=\vec b_i\alpha_1$ for all $i\geq j$. Since the
sequence $\vec b_i$ is periodic with period $p$, we conclude $\vec
b_i\alpha_1$ is independent of $i$ and in particular coincides with
$\vec b_p\alpha_1= x_{\ell+1}$, as required.

\subsubsection*{Subcase 2} Suppose that $(\vec b_i\alpha_1){}_{\vec a}\wh f<_{\R}
\vec b_i\alpha_1$ for all $1\leq i\leq p$.  Then since $\vec x_if =
(\vec b_i){}_{\vec a}\wh f\cdot \vec af\in \flagca$, we deduce
$(\vec b_i)_{\vec a}\wh f\in\flagca$. Therefore, every entry of
$(\vec b_i){}_{\vec a}\wh f$ is $<_{\J}$-below $\vec b_i\alpha_1\L
s=\vec z\omega$ (see \eqref{superequation}). Since $|\vec
z|=\ell+1$,
\[\vec b_{i+1}\alpha_1 = \vec x_if\red\tau_{\ell+1} = ((\vec b_i){}_{\vec
a}\wh f\cdot \vec af)\red\tau_{\ell+1} = ((\vec b_i){}_{\vec a}\wh
f\cdot \vec z)\red\tau_{\ell+1}=\vec z\omega\] In particular, $\vec
b_i\alpha_1$ is independent of $i$, and so taking $i=p$ shows that
$\vec b_i\alpha_1 = \vec b\alpha_1=x_{\ell+1}$, as desired.

\subsection*{Case 3} We now arrive at the final case: when
$|\vec z|=\ell$, i.e.\ $\vec z=\vec a$.  This case does not arise
when $\ell=0$ since $\varepsilon f\rho = \ov f\rho \neq
\varepsilon$.  So assume from now on $\ell\geq 1$. This is the only
case that makes use of the definition of $\bigwrf$. Observe
\begin{equation}\label{bigequationnew}
x_{\ell}=\vec a\omega = \vec z\omega = \vec af\red\omega = \vec
af\omega= (\vec a\wh f\cdot \ov f)\omega = \vec a\wh f\omega.
\end{equation}
Since $\vec x_i = \vec b_i\cdot \vec a$ is a flag, we have the
important formula
\begin{equation}\label{anotherusefuleq}
\vec b_{i+1}\alpha_1 = \vec x_if\red\tau_{\ell+1} = ((\vec
b_i){}_{\vec a}\wh f\cdot \vec af)\red \tau_{\ell+1} = (\vec
b_i){}_{\vec a}\wh f\rho\alpha_1\L (\vec b_i\alpha_1){}_{\vec a}\wh
f
\end{equation}
where the last equality uses $|\vec af\rho| = |\vec z|=\ell$, while
the $\L$-equivalence comes from the second item of
Lemma~\ref{reductionmap}.  The following subclaim will be used to
seal the rest of the proof.

\begin{Subclaim}\label{secondclaim}
There do not exist $m>i\geq 0$ such that $\vec b_{m}\alpha_1<_{\J}
\vec b_i\alpha_1.$
\end{Subclaim}
\noindent Indeed, suppose $\vec b_{m}\alpha_1<_{\J} \vec
b_i\alpha_1$. Then by \eqref{anotherusefuleq}, we have \[\vec
b_{m+1}\alpha_1\L (\vec b_m\alpha_1){}_{\vec a}\wh f\leq_{\R} \vec
b_m\alpha_1<_{\J} \vec b_i\alpha_1.\]  Continuing, we see that $\vec
b_n\alpha_1<_{\J} \vec b_i\alpha_1$ for all $n\geq m$.  But the
sequence $\vec b_n$ is periodic with period $p$, so choosing an
appropriate $n$ yields $\vec b_i<_{\J} \vec b_i$. This contradiction
establishes Subclaim~\ref{secondclaim}.

Now we are in a position to prove that the $\L$ in
\eqref{anotherusefuleq} is really an equality:
\begin{equation}\label{evenmoreusefulequation}
\vec b_{i+1}\alpha_1 =(\vec b_i\alpha_1){}_{\vec a}\wh f.
\end{equation}
Indeed, by \eqref{anotherusefuleq} $\vec b_{i+1}\alpha_1\L (\vec
b_i\alpha_1){}_{\vec a}\wh f\leq_{\R} \vec b_i\alpha_1$.
Subclaim~\ref{secondclaim}  then implies
\begin{equation}\label{whyname}
(\vec b_i\alpha_1){}_{\vec a}\wh f\R \vec b_i\alpha_1>_{\J} \vec
b_i\tau_2\geq _{\R} (\vec b_i){}_{\vec a}\wh f\tau_2
\end{equation}
since $\vec b_i$ is a flag.  From \eqref{anotherusefuleq} and
\eqref{whyname} it follows that indeed \[\vec b_{i+1}\alpha_1 =
(\vec b_i){}_{\vec a}\wh f\rho\alpha_1 = (\vec b_i\alpha_1){}_{\vec
a}\wh f.\]

Let us now prove by induction on $0\leq i\leq p-1$ that $\vec
b_i\alpha_1 = x_{\ell+1}$ (where $\vec b_0=\vec b$).  The case $i=0$
is trivial. Suppose that the statement is true for $i$ with $0\leq
i\leq p-2$.  Then \eqref{evenmoreusefulequation} and induction
implies
\[\vec b_{i+1}\alpha_1 = (\vec b_i\alpha_1){}_{\vec a}\wh f\leq_{\R}
\vec b_i\alpha_1 = x_{\ell+1}.\]  Since $\vec
b_{i+1}\alpha_1\nless_{\J} \vec b_i\alpha_1$ (by
Subclaim~\ref{secondclaim}), we must in fact have $\vec
b_{i+1}\alpha_1\R x_{\ell+1}$. Putting together $(\vec
b_i\alpha_1\cdot \vec a)\wh f = (\vec b_i\alpha_1){}_{\vec a}\wh
f\cdot \vec a\wh f$ with \eqref{bigequationnew} and
\eqref{evenmoreusefulequation} yields
\[(x_{\ell+1},x_{\ell},\ldots,x_1)\wh f = (\vec
b_{i+1}\alpha_1,x_{\ell},y_{\ell-1},\ldots,y_1).\]  Since
$x_{\ell+1}\R \vec b_{i+1}\alpha_1$, the Zeiger property
(Lemma~\ref{zeigerprop}) yields $b_{i+1}\alpha_1 = x_{\ell+1}$.
 This completes the induction
that $\vec b_i\alpha_1=x_{\ell+1}$ for all $i$ and thereby finishes
the proof of Claim~\ref{keyclaim}. Theorem~\ref{isaperiodic} is now
proved.
\end{proof}

\section{Blowup operators}
  Fix a finite semigroup $T$, a set of primes $\pi$ and set $S=\CP$.
The salient idea underlying the remainder of the proof, is to
construct a retraction \mbox{$\wh B:\flagc\to \flagca$} belonging to
$\bigwrf$. One then ``conjugates'' the action of the generators of
the Rhodes expansion on $\flag$ by this retract to get an action on
$\flaga$ belonging to $\suppera$.

 Let us write $s\leq_{\H} t$ if both
$s\leq_{\L} t$ and $s\leq_{\R} t$.

\begin{Def}[Blowup operator]\label{blowupoperator}
A \emph{preblowup operator} on $S$ is a function $\Hop:S\to S$
satisfying the following properties:
\begin{enumerate}
\item $s\Hop = s$ if $s$ is $\pi'$-free;
\item $s\Hop<_{\H} s$ if $s$ is not $\pi'$-free;
\item $s\subseteq s\Hop$ (the ``blow up'');
\item There exists a function $m:S\to S$, written $s\mapsto m_s$, such that
$s\Hop = sm_s$ and $m_s=m_{s'}$ whenever $s\L s'$.
\end{enumerate}
An idempotent preblowup operator is called a \emph{blowup operator}.
\end{Def}

The element $m_s$ is called the \emph{right multiplier} associated
to $s$.

\begin{Lemma}\label{preblowupisemigroup}
The collection of preblowup operators on $S$ is a finite semigroup.
In particular, if there are any preblowup operators on $S$, then
there is a blowup operator on $S$.
\end{Lemma}
\begin{proof}
The first three conditions are obviously closed under composition.
If $\Hop$ and $\Hop'$ are preblowup operators with respective right
multipliers \mbox{$s\mapsto m_s$} and $s\mapsto n_s$, then
$s\Hop\Hop' = sm_sn_{sm_s}$.  If $s\L s'$, then $m_s=m_{s'}$ and
$sm_s\L s'm_s$. Therefore, $n_{sm_s} = n_{s'm_s}=n_{s'm_{s'}}$. This
shows that $\Hop\Hop'$ is a preblowup operator with $m_sn_{sm_s}$ as
the right multiplier associated to $s$. The final statement follows
from the existence of idempotents in non-empty finite semigroups.
\end{proof}

The next proposition collects some elementary properties of blowup
operators.  For the first item, the reader should consult
Definition~\ref{definechecks}.
\begin{Prop}\label{blowupprops}
Let $\Hop:S\to S$ be a blowup operator. Then:
\begin{enumerate}
\item $\Hop\in \check S$;
\item The image of $\Hop$ is the set of $\pi'$-free elements of $S$;
\item Suppose $y\leq_{\L} s$. Then $y\subseteq ym_s$;
\item If $s$ is $\pi'$-free and $y\leq_{\L} s$, then $y=ym_s$
\end{enumerate}
\end{Prop}
\begin{proof}
First we check $B\in \check S$.  Since $s\Hop = sm_s$, clearly
$s\Hop \leq_{\R} s$. If $s\L t$, then $m_s=m_t$ and so $sB = sm_s\L
tm_s=tm_t=tB$. If $s\Hop \R s$, then by the first and second items
in the definition of a blowup operator, $s$ is a  $\pi'$-free and
$s\Hop = s = sI$, so we may take $s_{\Hop} =I$ in the third item of
Definition~\ref{definechecks}.

For the second item, observe that $\Hop$ fixes an element $s$ if and
only if $s$ is $\pi'$-free. Since $\Hop$ is idempotent, it image is
its fixed-point set.   Turning to the third item, write $y = zs$
with $z\in S^I$. Then we have
\[y=zs\subseteq z(s\Hop) = zsm_s = ym_s,\] as required.  For the
final item, we have $s=s\Hop=sm_s$.  Since $y=zs$, some $z\in S^I$,
we have $ym_s = zsm_s =zs=y$.  This completes the proof.
\end{proof}

For the rest of this section we assume the existence of a blowup
operator $\Hop$ on $S$; a construction appears in the next section.
We proceed to define an ``extension'' $\wh\Hop$ of $\Hop$ to
$\flagc$. Recall that $\Delta_{s}$ is the diagonal operator in
$\bigwr$ corresponding to $s$ \eqref{diagonal}.

\begin{Def}[$\wh\Hop$]\label{extendblowup}
Define $\wh\Hop:S^*\to S^*$ recursively by
\begin{itemize}
\item $\varepsilon\wh\Hop=\varepsilon$
\item $(\vec b\cdot s)\wh\Hop = (\vec b\Delta_{m_s})\wh\Hop\cdot s\Hop$.
\end{itemize}
\end{Def}
This recursive definition is known as the Henckell formula.

Since $s\Hop\leq_{\L}s$, we obtain the following lemma.
\begin{Lemma}\label{firstdropsinL}
If $\vec x\neq \varepsilon$, then $\vec x\wh\Hop\alpha_1\leq_{\L}
\vec x\alpha_1$.
\end{Lemma}

We retain the notation from the previous section for the next
proposition.
\begin{Prop}\label{belongstobigwrf}
The map $\wh\Hop$ belongs to $\bigwrf$.  Moreover,
$\flagc\wh\Hop=\flagca$ and $\wh\Hop|_{\flagc}$ is idempotent.
\end{Prop}
\begin{proof}
Since $\Hop\in \check S$ by Proposition~\ref{blowupprops} and
$\Delta_{m_s}\in \bigwr$, it is immediate from the recursive
definition that $\wh\Hop\in \bigwr$. Next we verify that
$\flagc\wh\Hop\subseteq \flagca$ by induction on length. The base
case is trivial.  In general, $(\vec b\cdot x)\wh\Hop = (\vec
b\Delta_{m_x})\wh\Hop\cdot x\Hop$.  Since $\Delta_{m_x}$ preserves
$\flagc$, by induction $(\vec b\Delta_{m_x})\wh\Hop\in \flagca$.
Since $x\Hop$ is $\pi'$-free (Proposition~\ref{blowupprops}), if
$\vec b=\varepsilon$ we are done. Otherwise, let $x_1$ be the first
entry of $\vec b$. Then Lemma~\ref{firstdropsinL} shows that $(\vec
b\Delta_{m_x})\wh\Hop\alpha_1\leq_{\L} x_1m_x$.  As $x\Hop = xm_x$
and $x_1\leq_{\L} x$, we see that $x_1m_x\leq_{\L} xm_x$ and so
$(\vec b\Delta_{m_x})\wh\Hop\cdot x\Hop$ belongs to $\flagca$.

We show by induction on length that $\wh\Hop$ fixes $\flagca$, the
case of length $0$ being trivial.  If $\vec x=\vec b\cdot
x\in\flagca$, then $\vec x\wh\Hop = (\vec b\Delta_{m_x})\wh\Hop\cdot
x\Hop$.  But Proposition~\ref{blowupprops}, together with the fact
that $x$ is $\pi'$-free and $\vec x$ is an $\L$-chain, implies
$xB=x$ and $\vec b\Delta_{m_x}=\vec b$. So a simple induction yields
$\vec x\wh\Hop =\vec x$.  We conclude $\wh\Hop|_{\flagc}$ is
idempotent.

Finally, we must verify $\wh\Hop\in \bigwrf$. We proceed by
induction on length, the cases of length $0$ and $1$ being vacuously
true. Suppose $(x_2,x_1)\wh\Hop =(y_2,y_1)$ with $x_1\R y_1$ and
$x_2\R y_2$.  Then $y_1 = x_1\Hop = x_1m_{x_1}$.  On the other hand,
$x_2\R y_2=(x_2m_{x_1})\Hop\leq_{\R} x_2m_{x_1}\leq_{\R} x_2$.  Thus
$x_2m_{x_1}\R (x_2m_{x_1})\Hop$ and so $x_2m_{x_1}$ is $\pi'$-free.
Therefore $y_2=(x_2m_{x_1})\Hop=x_2m_{x_1}$.  Thus $y_1=x_1m_{x_1}$
and $y_2=x_2m_{x_2}$, showing that the condition in
Definition~\ref{defbigwrf} is satisfied. Suppose now $n>2$ and that
$(x_n,x_{n-1},\ldots,x_1)\wh\Hop = (y_n,y_{n-1},\ldots,y_1)$ with
$x_i\R y_i$ for $i=n-1,n$.  Then
\[(y_n,y_{n-1},\ldots,y_1)=((x_n,x_{n-1},\ldots,x_2)\Delta_{m_{x_1}})\wh\Hop\cdot
x_1\Hop.\]  Now $x_i\geq_{\R} x_im_{x_1}\geq_{\R} y_i\R x_i$, for
$i=n-1,n$.  Therefore, $x_im_{x_1}\R y_i$, $i=n-1,n$.  Induction
provides $s'\in S^I$ with $x_im_{x_1}s' = y_i$, $i=n-1,n$. Taking
$s=m_{x_1}s'$ yields $x_is=y_i$, for $i=n-1,n$, completing the
proof.
\end{proof}

Another crucial property of $\wh\Hop$ is that it ``blows up
$\L$-chains''.
\begin{Prop}\label{blowupstring}
Let $\vec x=(x_n,\ldots,x_1)\in \flagc$ and set $\vec x\wh\Hop =
(y_n,\ldots,y_1)$.  Then $x_i\subseteq y_i$ for $i=1,\ldots,n$.
\end{Prop}
\begin{proof}
The proof is by induction on $n$.  For $n=0$, the statement is
vacuously true.  In general, $\vec x\wh\Hop =
(x_nm_{x_1},\ldots,x_2m_{x_1})\wh\Hop\cdot x_1\Hop$.  By the
definition of a blowup operator $x_1\subseteq x_1\Hop$.  Since $\vec
x$ is an $\L$-chain, Proposition~\ref{blowupprops} shows
$x_i\subseteq x_im_{x_1}$ for $i=2,\ldots, n$.  Induction yields
$x_im_{x_1}\subseteq y_i$ for $i=2,\ldots,n$, establishing that
$x_i\subseteq y_i$.
\end{proof}

Recall that if $s\in S$, then $(\Delta_s,s)$ denotes the element of
$\mathscr C$ that acts by $\vec x(\Delta_s,s) = \vec x\Delta_s\cdot
s$.

\begin{Prop}\label{commutewithred}
The equalities $\red(\Delta_s,s)\red = (\Delta_s,s)\red$ and
$\red\wh\Hop\red =\wh\Hop\red$ hold.
\end{Prop}
\begin{proof}
Consider first $(\Delta_s,s)$.  Suppose $(x_n,\ldots,x_1)\in \flagc$
with $x_{i+1}\L x_i$.  Then $(x_n,\ldots,x_1)(\Delta_s,s) =
(x_ns,\cdots, x_1s,s)$ has $x_{i+1}s\L x_is$.  It follows that
applying first the elementary reduction $(x_{i+1},x_i)\to x_{i+1}$
and then $(\Delta_s,s)$ is the same as applying first $(\Delta_s,s)$
and then the elementary reduction \mbox{$(x_{i+1}s,x_is)\to
x_{i+1}s$}. We conclude $\red(\Delta_s,s)\red = (\Delta_s,s)\red$.

Let us now turn to $\wh\Hop\red$.  We show by induction on $i$ that
if a string \mbox{$\vec x=(x_n,\ldots,x_1)$} with $n>i$ admits an
elementary reduction $(x_{i+1},x_i)\to x_{i+1}$ and $\vec x\wh\Hop =
(y_n,\ldots,y_1)$, then $(y_{i+1},y_i)\to y_{i+1}$ is an elementary
reduction.  The base case is $i=1$, i.e.\ $x_1\L x_2$.  Then
\[(x_n,\ldots,x_2,x_1)\wh\Hop =
((x_n,\ldots,x_3)\Delta_{m_{x_1}m_{x_2m_{x_1}}})\wh\Hop\cdot\left((x_2m_{x_1})\Hop,
x_1m_{x_1}\right)\] By the definition of a blowup operator, $x_1\L
x_2$ implies that $m_{x_1}=m_{x_2}$.  Now $y_1=x_1m_{x_1}$.  Since
$x_2m_{x_1}=x_2m_{x_2} = x_2\Hop$ and $\Hop$ is idempotent, we see
that $y_2 = x_2m_{x_1}$. Now $x_2\L x_1$ implies $x_2m_{x_1}\L
x_1m_{x_1}$.  We conclude $y_2\L y_1$, as required. If $i>1$, then
we use that
\[(y_n,\ldots,y_1)=(x_nm_{x_1},\ldots,x_2m_{x_1})\wh\Hop\cdot
x_1\Hop.\]  Since $x_{i+1}m_{x_1}\L x_im_{x_1}$, the induction
hypothesis gives  $(y_{i+1},y_i)\to y_{i+1}$ is an elementary
reduction.  This completes the induction.  It is then immediate that
$\red\wh\Hop\red = \wh\Hop\red$.
\end{proof}

For the next proposition, the reader is referred to
Definition~\ref{CA}.

\begin{Prop}\label{inCA}
If $s\in S$, then $(\Delta_s,s)\wh\Hop\in \mathscr C^{\pi}$.
\end{Prop}
\begin{proof}
We saw in the proof of Proposition~\ref{CAissemi} that the set of
transformations $f$ satisfying $\red f\red =f\red$ is a semigroup.
As $\wh\Hop:\flagc\to \flagca$ (Proposition~\ref{belongstobigwrf}),
it suffices by Proposition~\ref{commutewithred}, to show that
$(\Delta_s,s)\wh\Hop\in \mathscr C$.  Both $(\Delta_s,s)$ and
$\wh\Hop$ leave $\flagc$ invariant. Now we obtain
\[\vec x(\Delta_s,s)\wh\Hop = (\vec x\Delta_s\cdot s)\wh\Hop = (\vec
x\Delta_s\Delta_{m_s})\wh\Hop\cdot s\Hop=(\vec
x\Delta_{sm_s})\wh\Hop\cdot s\Hop\] and $\Delta_{sm_s}\wh\Hop\in
\bigwrf$ by Proposition~\ref{belongstobigwrf}.  This establishes
$(\Delta_s,s)\wh\Hop\in \mathscr C$, completing the proof.
\end{proof}

\begin{Rmk}
Let $S$ be any finite monoid and $\Hop:S\to S$ any idempotent
operator satisfying (1), (2) and (4) of
Definition~\ref{blowupoperator}. Then one can define $\wh\Hop:S^*\to
S^*$ as per Definition~\ref{extendblowup} and Proposition~\ref{inCA}
will still hold.  Condition (3) is just used for
Proposition~\ref{blowupstring} and to construct the relational
morphism below.
\end{Rmk}

If $(X,M)$ and $(Y,N)$ are faithful transformation semigroups, then
a \emph{relational morphism} $\p:(X,M)\to (Y,N)$ is a fully defined
relation $\p:X\to Y$ such that, for each $m\in M$, there exists
$\til m\in N$ such that $y\pinv m\subseteq y\til m\pinv$ for all
$y\in Y$. If $\p:(X,M)\to (Y,N)$ is a relational morphism of
faithful transformation semigroups, then the \emph{companion
relation} $\til\p:M\to N$ is defined by $m\til\p = \{n\in N\mid
y\pinv m\subseteq yn\pinv, \forall y\in Y\}$.  It is well known that
$\til\p$ is a relational morphism~\cite{Eilenberg}.

We define a relational morphism of faithful transformation
semigroups $\p:(T^I,T)\to (\flaga,\suppera)$ as follows: we set $I\p
= \varepsilon$, while, for $t\in T$, we define $t\p = \{\vec x\in
\flaga \mid t\in \vec x\omega\}$.  Notice that $\pinv$ coincides
with $\omega$ on non-empty strings.

\begin{Lemma}\label{definetherelmorph}
The relation $\p:T^I\to \flaga$ gives rise to a relational morphism
$\p:(T^I,T)\to (\flaga,\suppera)$.
\end{Lemma}
\begin{proof}
Since $t\in \{t\}\subseteq \{t\}\Hop$ and $\{t\}\Hop\in \flaga$, it
follows that $\p$ is fully defined.  Let $t\in T$.  We set $\til t =
(\Delta_{\{t\}},\{t\})\wh\Hop\red$.  Proposition~\ref{inCA} shows
 $\til t\in \suppera$.   We need to prove $\vec x\p\inv
t\subseteq \vec x\til t\p\inv$. If $\vec x=\varepsilon$, then $\vec
x(\Delta_{\{t\}},\{t\})\wh\Hop\red = \{t\}\Hop$. So
$\varepsilon\pinv t=\{I\}t=\{t\} \subseteq \{t\}\Hop=\varepsilon
\til t\pinv$.

If $\vec x\neq\varepsilon$, then we need to show $\vec x\omega
t\subseteq \vec x\til t\omega$. Abusing notation, we identify
$\{t\}$ with $t$. Then we have $\vec x(\Delta_t,t)\wh\Hop\red =
\left((\vec x\Delta_t\cdot t)\wh\Hop\right)\red$.
Lemma~\ref{reductionmap} tells us $\left((\vec x\Delta_t\cdot
t)\wh\Hop\right)\red\omega = (\vec x\Delta_t\cdot t)\wh\Hop\omega$.
An application of Proposition~\ref{blowupstring} yields \[(\vec
x\Delta_t\cdot t)\wh\Hop\omega\supseteq (\vec x\Delta_t\cdot
t)\omega = \vec x\omega t,\] completing the proof.
\end{proof}

\begin{Prop}\label{blowupcomputes}
The companion relation $\til\p:T\to \suppera$ satisfies the
inequality $f\til\p\inv\subseteq \varepsilon f\omega\in \CP$.
\end{Prop}
\begin{proof}
Let $t\in f\til\p\inv$.  Then $t = It \in \varepsilon\pinv
t\subseteq \varepsilon f\pinv = \varepsilon f\omega\in \CP$, as
$\varepsilon f\neq \varepsilon$.  Hence $f\til\p\inv\subseteq
\varepsilon f\omega$, as required.
\end{proof}

\begin{Cor}\label{almostdone}
If $\CP$ admits a blowup operator, then Theorem~\ref{henckellmain}
holds.  That is, $\pl {\barGpi} T = \{X\subseteq T\mid X\subseteq
Y\in \CP\}$.
\end{Cor}
\begin{proof}
We already know $\CP\subseteq \mathsf {PL}_{\barGpi}(T)$. Since
$\suppera$ is $\pi'$-free by Theorem~\ref{isaperiodic}, we have that
each $\barGpi$-pointlike set is contained in $f\til\p\inv$ for some
$f\in \suppera$.  An application of Proposition~\ref{blowupcomputes}
then completes the proof.
\end{proof}

\section{Construction of the blowup operator}
We continue to work with our fixed finite semigroup $T$ and to
denote $\CP$ by $S$.  Our task now consists of constructing a blowup
operator for $S$.  By Lemma~\ref{preblowupisemigroup}, it suffices
to construct a preblowup operator.  Our approach is a variation on
Henckell's~\cite{Henckell}, which leads to a shorter proof.   For
this purpose, we need to use  Sch\"utzenberger groups.  We retain
the notation for Sch\"utzenberger groups introduced in
Section~\ref{schutzsec}.

For each non-$\pi'$-free $\L$-class $L$, fix an $\H$-class $H_L$ of
$L$ and a prime power order element $g_L\in \til \Gamma_R(H_L)$
representing an element of $\Gamma_R(H_L)$ of prime order $p\in
\pi'$
 (c.f.~Lemma~\ref{primelift}).
We are now prepared to define our preblowup operator $\Hop$. If
$s\in S$ is a $\pi'$-free element, define $m_s=I$. If $s$ is not
$\pi'$-free, define $m_s=g_{L_s}^{\omega+\ast}\in S$. Notice that
$g_{L_s}^{\omega+\ast} = \bigcup_{n\geq 1} g_{L_s}^n\supset g_{L_s}$
since $g_{L_s}$ is a group element.  Define an operator $\Hop:S\to
S$ by $s\Hop = sm_s$.

\begin{Prop}
The operator $\Hop$ is a preblowup operator.
\end{Prop}
\begin{proof}
If $s$ is $\pi'$-free, then $m_s=I$ and $s\Hop = sm_s=s$.  The
fourth item of Definition~\ref{blowupoperator} is clearly satisfied
by construction. We turn now to the third item.  If $s$ is
$\pi'$-free, then trivially $s\subseteq s\Hop$.  If $s$ is not
$\pi'$-free, then since $g_{L_s}\in \til\Gamma_R(H_s)$ by
Proposition~\ref{L-classindep}, we have $sg_{L_s}^{\omega} = s$.  As
$g_{L_s}^{\omega}\subseteq g_{L_s}^{\omega+\ast}$,
\[s=sg_{L_s}^{\omega}\subseteq sg_{L_s}^{\omega+\ast} =
sm_s=s\Hop,\] as required.  Finally we turn to the second item of
Definition~\ref{blowupoperator}.  Suppose that $s\in S$ is not
$\pi'$-free.  It is immediate from the definition  $s\Hop =
sm_s\leq_{\R} s$. Let $\gamma:\Gamma_R(H_s)\to \Gamma_L(H_s)$ be the
anti-isomorphism given by $sg = g\gamma s$ for $g\in \Gamma_R(H_s)$.
Choose, using Lemma~\ref{primelift}, an element $x\in \til
\Gamma_L(H_s)$ of order a power of $p$ so that $x$ maps to
$g_{L_s}\gamma$ in $\Gamma_L(H_s)$ (where we view $g_{L_s}$ as an
element of $\Gamma_R(H_s)$ using Proposition~\ref{L-classindep} and
the projection). Then we have $\bigcup_{n\geq
1}x^n=x^{\omega+\ast}\in S$. We calculate $s\Hop$ as follows:
\begin{align*}
s\Hop = sg_{L_s}^{\omega+\ast} = s\bigcup_{n\geq 1} g_{L_s}^n =
\bigcup_{n\geq 1}sg_{L_s}^n  = \bigcup_{n\geq 1}(g_{L_s}\gamma)^ns =
\bigcup_{n\geq 1}x^ns = x^{\omega+\ast}s.
\end{align*}
We conclude $s\Hop\leq_{\L} s$ and thus $s\Hop\leq_{\H} s$. To
establish  $s\Hop<_{\H} s$,  observe, using \eqref{eq:omega+*},
$s\Hop g_{L_s} =sg_{L_s}^{\omega+\ast}g_{L_s} =
sg_{L_s}^{\omega+\ast} = s\Hop$. Since $g_{L_s}$ represents a
non-trivial element of $\Gamma_R(H_s)$
(c.f.~Proposition~\ref{L-classindep}) and $(H_s,\Gamma_R(H_s))$ is a
regular permutation group, we deduce that $s\Hop\notin H_s$. This
concludes the proof that $s\Hop <_{\H} s$ when $s$ is not
$\pi'$-free. Therefore, $\Hop$ is a preblowup operator, as required.
\end{proof}

In light of Corollary~\ref{almostdone}, we have now established
Theorem~\ref{henckellmain}.

\begin{Rmk}
The first and second authors believe that one can make this whole
approach work without blowing up null elements.
\end{Rmk}

\bibliographystyle{abbrv}
\bibliography{standard}
\end{document}